\documentclass[a4paper,reqno,11pt, oneside]{amsart}

\usepackage[all]{xy}
\xyoption{poly}
\usepackage{amssymb}
\usepackage{amstext}
\usepackage{amsmath}
\usepackage{amscd}
\usepackage{amsthm}
\usepackage{amsfonts}
\usepackage{latexsym}
\usepackage{graphicx}
\usepackage{color}
\usepackage{geometry}
\usepackage{mathrsfs}

\makeatletter
  
  \@addtoreset{equation}{section}
\makeatother

\makeatletter
\@namedef{subjclassname@2020}{%
  \textup{2020} Mathematics Subject Classification}
\makeatother

\newcommand{\calF}{\mathcal{F}}

\newcommand{\calO}{\mathcal{O}}
\newcommand{\calS}{\mathcal{S}}
\newcommand{\calV}{\mathcal{V}}

\newcommand{\scS}{\mathscr S}

\newcommand{\NN}{\mathbb{N}}

\newcommand{\PP}{\mathbb{P}}

\def\opn#1#2{\def#1{\operatorname{#2}}} 
\opn\rank{rank} \opn\mnull{null} \opn\Iso{Iso} \opn\Sw{Sw}
\opn\type{type}

\def\Qcoh{\operatorname{\mathsf{Qcoh}}}

\def\GrMod{\operatorname{\mathsf{GrMod}}}
\def\Tors{\operatorname{\mathsf{Tors}}}

\def\QGr{\operatorname{\mathsf{QGr}}}

\def\Projn{\operatorname{\mathsf{Proj_{nc}}}}
\def\e{\varepsilon}
\opn\Proj{Proj}

\def\rnum#1{\expandafter{\romannumeral #1}}
\def\Rnum#1{\uppercase\expandafter{\romannumeral #1}}

\newtheorem{thm}{Theorem}[section]
\newtheorem{cor}[thm]{Corollary}
\newtheorem{lem}[thm]{Lemma}
\newtheorem{prop}[thm]{Proposition}

\newtheorem{conj}[thm]{Conjecture}

\theoremstyle{definition}
\newtheorem{dfn}[thm]{Definition}
\newtheorem{ex}[thm]{Example}

\newtheorem{rem}[thm]{Remark}

\begin{document}

\title [Combinatorial classification of $(\pm 1)$-skew projective spaces]
{Combinatorial classification of \\ $(\pm 1)$-skew projective spaces}

\author{Akihiro Higashitani}
\address{Department of Pure and Applied Mathematics, 
Graduate School of Information Science and Technology, 
Osaka University, 
1-5, Yamadaoka, Suita, Osaka 565-0871, Japan}
\email{higashitani@ist.osaka-u.ac.jp}

\author{Kenta Ueyama}
\address{Department of Mathematics,
Faculty of Education,
Hirosaki University,
1 Bunkyocho, Hirosaki, Aomori 036-8560, Japan}
\email{k-ueyama@hirosaki-u.ac.jp}

\keywords{
$(\pm 1)$-skew projective space,
$(\pm 1)$-skew polynomial algebra,
noncommutative projective scheme,
point variety,
switching of graphs,
simplicial complex}

\subjclass[2020]{14A22, 16S38, 16W50, 05C76, 13F55}

\begin{abstract}
The noncommutative projective scheme $\operatorname{\mathsf{Proj_{nc}}} S$ of a $(\pm 1)$-skew polynomial algebra $S$ in $n$ variables is considered to be a $(\pm 1)$-skew projective space of dimension $n-1$.
In this paper, using combinatorial methods, we give a classification theorem for $(\pm 1)$-skew projective spaces.
Specifically, among other equivalences, we prove that $(\pm 1)$-skew projective spaces $\operatorname{\mathsf{Proj_{nc}}} S$ and $\operatorname{\mathsf{Proj_{nc}}} S'$ are isomorphic if and only if certain graphs associated to $S$ and $S'$ are switching (or mutation) equivalent.
We also discuss invariants of $(\pm 1)$-skew projective spaces from a combinatorial point of view.
\end{abstract}

\maketitle 


\section{Introduction}
Since the noncommutative projective scheme of a Koszul AS-regular algebra is considered as a noncommutative projective space,
one of the most important projects in noncommutative algebraic geometry is the classification of Koszul AS-regular algebras.

Koszul AS-regular algebras of dimension $3$ were classified by Artin and Schelter \cite{AS}, 
and by Artin, Tate and Van~den~Bergh \cite{ATV} using a geometric approach, introducing the notion of a point scheme. 

On the other hand, the complete classification of Koszul AS-regular algebras of dimension $4$ or higher is a very difficult problem.
To address this problem, researchers have been studying many special classes of $4$-dimensional Koszul AS-regular algebras.
The most famous class is Sklyanin algebras of dimension $4$, whose ring-theoretic properties were studied by Smith and Stafford \cite{SS}.
Some variants of Sklyanin algebras were studied by Stafford \cite{Staf}.
Normal extensions of AS-regular algebras of dimension $3$ were studied by Le Bruyn, Smith and Van~den~Bergh \cite{LSV}.
Koszul AS-regular algebras of dimension $4$ having special point schemes or special line schemes were studied by
Shelton, Tingey, Vancliff, Van~Rompay, Willaert, etc.\ \cite{ST}, \cite{SV1}, \cite{SV2}, \cite{VVW}, \cite{Va}.
Moreover, it is known that graded skew Clifford algebras \cite{CV} and double Ore extensions \cite{ZZ1}, \cite{ZZ2} are very useful for constructing and studying $4$-dimensional Koszul AS-regular algebras.
In dimension $5$, AS-regular algebras that are quadratic but not Koszul were constructed by  Fl\o ystad and Vatne \cite{FV}, and by Li and Wang \cite{LW}.

The purpose of this paper is to classify a special class of $n$-dimensional Koszul AS-regular algebras, called standard graded $(\pm 1)$-skew polynomial algebras in $n$ variables.

Throughout, let $k$ be an algebraically closed field of characteristic not $2$.
\begin{dfn}\label{dfn:pm1}
A \emph{standard graded $(\pm 1)$-skew polynomial algebra} in $n$ variables is a graded algebra
\[ S_{\e} = k\langle x_1, \dots, x_n \rangle /(x_ix_j -\e_{ij} x_jx_i \mid 1\leq i,j \leq n) \]
where
\begin{itemize}
\item $\deg x_i=1$ for all $1\leq i \leq n$, and
\item $\e =(\e_{ij}) \in M_n(k)$ is a symmetric matrix such that $\e_{ii}=1$ for all $1\leq i \leq n$ and $\e_{ij}=\e_{ji} \in \{1, -1\}$ for all $1\leq i<j \leq n$.
\end{itemize}
The number $n$ is called the \emph{dimension of $S_\e$}, and is denoted by $\dim S_\e$.
\end{dfn}

The most significant reason for focusing on the study of $(\pm 1)$-skew polynomial algebras is that not only geometric but also combinatorial approaches work very effectively.

For a standard graded $(\pm 1)$-skew polynomial algebra $S_{\e}$ in $n$ variables,
\begin{enumerate}
\item let $G_\e$ be the graph associated to $S_{\e}$ defined in Definition \ref{dfn:g},
\item let $\Sw(G_\e)$ be the switching graph of $G_\e$ in the sense of Godsil and Royle \cite{GR} (see Definition \ref{dfn:sw}),
\item let $\overline{S_\e}$ be the standard graded $(\pm 1)$-skew polynomial algebra in $2n$ variables defined in Definition \ref{dfn:overline},
\item let $\GrMod S_\e$ be the category of graded right $S_\e$-modules with degree-preserving $S_\e$-module homomorphisms,
\item let $\Projn S_\e := (\QGr S_\e, \calS_\e)$ be the noncommutative projective scheme of $S_\e$ in the sense of Artin and Zhang \cite{AZ} (see Definition \ref{dfn:nps}), 
\item let $\Gamma_\e$ be the point variety of $S_\e$ in the sense of Artin, Tate and Van den Bergh \cite{ATV} (see Definition \ref{dfn:ps}, Remark \ref{rem:pv}), and
\item let $\Delta_\e$ be the simplicial complex associated to $S_{\e}$ defined in Definition \ref{dfn:Del}.
\end{enumerate}

Our main result is the following theorem.

\begin{thm}[Theorem \ref{thm:main}]\label{thm:i.main}
Let $S_{\e}$ and $S_{\e'}$ be standard graded $(\pm 1)$-skew polynomial algebras. Then the following are equivalent.
\begin{enumerate}
\item $G_\e$ and $G_{\e'}$ are switching (or mutation) equivalent.
\item $\Sw(G_\e)$ and $\Sw(G_{\e'})$ are isomorphic as graphs.
\item $\overline{S_\e}$ and $\overline{S_{\e'}}$ are isomorphic as graded algebras.
\item $\GrMod S_\e$ and $\GrMod S_{\e'}$ are equivalent as categories.
\item $\Projn S_\e$ and $\Projn S_{\e'}$ are isomorphic as noncommutative projective schemes.
\item $\Gamma_\e$ and $\Gamma_{\e'}$ are isomorphic as varieties.
\item $\Delta_\e$ and $\Delta_{\e'}$ are isomorphic as simplicial complexes.
\end{enumerate}
\end{thm}

Note that Mori and the second author \cite[Theorem 6.20]{MU} proved that if $n \leq 6$, then the equivalences $(1) \Leftrightarrow (4)\Leftrightarrow (6)$ hold. Theorem \ref{thm:i.main} is a generalization of this work.

A noteworthy equivalence is $(1) \Leftrightarrow (5)$. 
If we consider the case $\e_{ij}=1$ for all $1 \leq i,j \leq n$,
then $S_\e = k[x_1, \cdots, x_n]$,
so it follows from Serre's theorem that
there exists an equivalence $F: \Qcoh \PP^{n-1} \to \QGr S_\e$ such that $F(\calO_{\PP^{n-1}}) \cong \calS_\e$,
where $\Qcoh \PP^{n-1}$ is the category of quasi-coherent sheaves on $\PP^{n-1}$ and $\calO_{\PP^{n-1}}$ is the structure sheaf.
Therefore, if $S_\e$ is a general standard graded $(\pm 1)$-skew polynomial algebra in $n$ variables,
then $\Projn S_\e$ can be considered as a $(\pm 1)$-skew analogue of $\PP^{n-1}$.
For this reason, we call it a \emph{$(\pm 1)$-skew projective space} of dimension $n-1$.
Thanks to $(1) \Leftrightarrow (5)$, we can completely classify $(\pm 1)$-skew projective spaces by a purely combinatorial operation on graphs, called switching.
In particular, we have the following consequence.

\begin{cor}
Let $a_n$ be the number of the switching equivalence classes of graphs on $n$ nodes.
Then the number of the isomorphism classes of $(\pm 1)$-skew projective spaces of dimension $n-1$ is equal to $a_n$.
\end{cor}

By \cite{MS}, we see that the values of $a_n$ are
\[
\renewcommand{\arraystretch}{1.2}
\begin{array} {c|ccccccccccccccc}
n &1&2&3&4&5&6&7&8&9&10&11&12&13&\cdots \\ \hline
a_n&1&1&2&3&7&16&54&243&2038&33120&1182004&87723296&12886193064&\cdots 
\end{array}
\]
(A002854 of OEIS \cite{OE}).

Note that our results depend heavily on the setting ``$(\pm 1)$-skew''.
For example, $(6) \Rightarrow (4)$ does not hold for general skew (not necessarily $(\pm 1)$-skew) polynomial algebras (see Remark \ref{rem:gsp}).
Moreover, the number of the isomorphism classes of point varieties for $(\pm 1)$-skew polynomial algebras in $n$ variables (which is equal to $a_n$) is different from the one for general skew polynomial algebras in $n$ variables (see Remark \ref{rem:psd}).

As a further remark on Theorem \ref{thm:i.main}, we mention that $(4) \Leftrightarrow (5)$ gives a partial affirmative answer to Vitoria's conjecture (see Conjecture \ref{conj:vi}).

In the last section, we discuss some invariants for the classification obtained in Theorem \ref{thm:i.main}.
Namely, we investigate invariants, such as dimension and type, of point varieties $\Gamma_\e$ using combinatorial methods.
(Note that invariants of point varieties $\Gamma_\e$ can be viewed as invariants of $(\pm 1)$-skew projective spaces $\Projn S_\e$ by Theorem \ref{thm:i.main}.)
We prove the following result.

\begin{thm}[{Corollary \ref{cor:dim}}]
Let $S_{\e}$ be a standard graded $(\pm 1)$-skew polynomial algebra.
Then the following numbers are equal:
\begin{enumerate}
\item the dimension of $\Gamma_\e$,
\item the dimension of $\Delta_\e$,
\item $\max\{ |F|-1 \mid$$F$ is a disjoint union of two cliques of $G_\e$ that are not connected by an edge$\}$,
\item the clique number of $\Sw(G_\e)$ minus one.
\end{enumerate}
\end{thm}

Furthermore, when the graph $G _\e$ is a disjoint union of complete graphs,
we give a formula for easily calculating the type of the point variety $\Gamma_\e$ (Corollary \ref{cor:type}).
Using this result, we show that type is not a complete invariant for point varieties (Example \ref{ex:cex}).

\section{Geometric aspects of skew polynomial algebras}
In this section, we discuss geometric aspects of skew (not necessarily $(\pm 1)$-skew) polynomial algebras.
Recall that $k$ denotes an algebraically closed field of characteristic different from $2$ throughout the paper.
\begin{dfn}
A \emph{standard graded skew polynomial algebra} in $n$ variables is a graded algebra
\[ S_{\alpha} = k\langle x_1, \dots, x_n \rangle /(x_ix_j -\alpha_{ij} x_jx_i \mid 1\leq i,j \leq n) \]
where
\begin{itemize}
\item $\deg x_i=1$ for all $1\leq i \leq n$, and
\item $\alpha =(\alpha_{ij}) \in M_n(k)$ is a matrix such that $\alpha_{ii}=\alpha_{ij}\alpha_{ji}=1$ for all $1\leq i,j \leq n$.
\end{itemize}
The number $n$ is called the \emph{dimension of $S_{\alpha}$}, and is denoted by $\dim S_{\alpha}$.
\end{dfn}
A standard graded $(\pm 1)$-skew polynomial algebra $S_\e$, defined in Definition \ref{dfn:pm1}, is clearly a special case of a standard graded skew polynomial algebra.
Since a standard graded skew polynomial algebra in $n$ variables is an $n$-iterated Ore extension of $k$,
we see that it is a (two-sided) noetherian Koszul AS-regular algebra of (global) dimension $n$ with Hilbert series $(1-t)^{-n}$  (see \cite{R} for basic information about AS-regular algebras).

\begin{prop}[{\cite[Lemma 2.3]{Vi}}]\label{prop:skewiso}
Let $S_{\alpha}$ and $S_{\alpha'}$ be standard graded skew polynomial algebras in $n$ variables.
Then $S_{\alpha} \cong S_{\alpha'}$ as graded algebras if and only if there exists a permutation $\sigma \in {\mathfrak S}_n$ such that $\alpha'_{ij}=\alpha_{\sigma(i)\sigma(j)}$ for $1 \leq i, j \leq n$.
\end{prop}

Let $S_{\alpha}$ be a standard graded skew polynomial algebra in $n$ variables.
First, we recall the definition of the noncommutative projective scheme of $S_{\alpha}$.
Let $\GrMod S_{\alpha}$ denote the category of graded right $S_{\alpha}$-modules and degree-preserving $S_{\alpha}$-module homomorphisms.
We define $\QGr S_{\alpha}$ to be the Serre quotient category $\GrMod S_{\alpha}/ \Tors S_{\alpha}$, where $\Tors S_{\alpha}$ is the full subcategory of $\GrMod S_{\alpha}$ consisting of direct limits of finite-dimensional modules.

\begin{dfn}[{\cite{AZ}}]\label{dfn:nps}
Let $S_{\alpha}$ be a standard graded skew polynomial algebra.
We call the pair 
\[\Projn S_{\alpha} :=(\QGr S_{\alpha}, \calS_{\alpha})\]
the \emph{noncommutative projective scheme of $S_{\alpha}$}, where $\calS_{\alpha}$ is the object in $\QGr S_{\alpha}$ that corresponds to the module $S_{\alpha} \in  \GrMod S_{\alpha}$.
We say that two noncommutative projective schemes $\Projn S_{\alpha}$ and $\Projn S_{\alpha'}$ are \emph{isomorphic} if there exists a $k$-linear equivalence functor $F: \QGr S_{\alpha} \to \QGr S_{\alpha'}$ such that $F(\calS_{\alpha}) \cong \calS_{\alpha'}$.
\end{dfn}

Note that if $\Projn S_{\alpha}$ and $\Projn S_{\alpha'}$ are isomorphic, then
$\dim S_{\alpha}=\dim S_{\alpha'}$ (see \cite[Lemma 1.5]{Vi}).

In noncommutative algebraic geometry, $\Projn S_{\alpha}$ is considered as a skew (or quantum) version of $\PP^{n-1}$.
Hence, we refer to $\Projn S_{\alpha}$ of a standard graded skew polynomial algebra $S_{\alpha}$ as a \emph{skew projective space}.
In particular, we refer to $\Projn S_{\e}$ of a standard graded $(\pm 1)$-skew polynomial algebra $S_\e$ as a \emph{$(\pm 1)$-skew projective space}.

As a special case of Zhang's work \cite{Z}, the following is well-known.

\begin{thm}[{\cite[Theorems 3.5 and 3.7]{Z}}] \label{thm:z}
Let $S_{\alpha}$ and $S_{\alpha'}$ be standard graded skew polynomial algebras.
If $\GrMod S_{\alpha}$ is equivalent to $\GrMod S_{\alpha'}$,
then $\Projn S_{\alpha}$ is isomorphic to $\Projn S_{\alpha'}$.
\end{thm}

The converse of Theorem \ref{thm:z} was stated as a conjecture by Vitoria \cite{Vi}.

\begin{conj}[{\cite[Conjecture 5.7]{Vi}}]\label{conj:vi}
Let $S_{\alpha}$ and $S_{\alpha'}$ be standard graded skew polynomial algebras.
If $\Projn S_{\alpha}$ is isomorphic to $\Projn S_{\alpha'}$, then
$\GrMod S_{\alpha}$ is equivalent to $\GrMod S_{\alpha'}$.
\end{conj}

Next, we recall the definition of the point scheme of $S_{\alpha}$.

\begin{dfn}
A graded module $M \in \GrMod S_{\alpha}$ is called a \emph{point module} if $M$ is a cyclic module, generated in degree $0$, with Hilbert series $(1-t)^{-1}$.
\end{dfn}

Let $V$ be a $k$-vector space spanned by $x_1, \dots, x_n$.
In other words, $V$ is the degree $1$ part of $S_{\alpha}$.
The dual vector space of $V$ is denoted by $V^*$.
If $M \in \GrMod S_{\alpha}$ is a point module, then $M$ can be presented as a quotient $S_{\alpha}/(g_1S_{\alpha}+ g_2S_{\alpha}+ \cdots + g_{n-1}S_{\alpha})$ with linearly independent $g_1, \dots, g_{n-1} \in V$
(see \cite[Theorem 3.8, Corollary 5.7]{Mcp}),
so we can associate it with a unique point $p_M := \calV(g_1, \dots, g_{n-1})$ in $\PP({V^*}) = \PP^{n-1}$.
Then the subset
\[ \Gamma_{\alpha} := \{ p_M \in \PP^{n-1} \mid M \in \GrMod S_{\alpha} \;\text{is a point module}\} \; \subset \PP^{n-1} \]
has a $k$-scheme structure by \cite{ATV}.

\begin{dfn}[{\cite{ATV}}]\label{dfn:ps}
The above scheme $\Gamma_{\alpha}$ is called the \emph{point scheme of $S_{\alpha}$}.
\end{dfn}

For a subset $X \subset [n]:= \{1,\dots  n\}$, we define the subspace
\[\PP(X):= \bigcap_{i \in [n]\setminus X} {\mathcal V}(x_i) \; \subset \PP^{n-1}.\]
Note that $\PP([n])=\PP^{n-1}$. 
If $X=\{i_1,\dots,i_s\}$, then $\PP(X)$ is denoted by $\PP(i_1,\dots,i_s)$.

Thanks to the following result due to Vitoria \cite{Vi} and independently Belmans, De~Laet and Le~Bruyn \cite{BDL},
we can compute the point scheme of $S_\alpha$ explicitly.

\begin{thm}[{\cite[Proposition 4.2]{Vi}, \cite[Theorem 1 (1)]{BDL}}]\label{thm.ps}
Let $S_\alpha = k\langle x_1, \dots, x_n \rangle/(x_ix_j-\alpha_{ij}x_jx_i)$ be a standard graded skew polynomial algebra.
Then the point scheme of $S_\alpha$ is given by
\[ \Gamma_{\alpha}= \bigcap_{\substack{1\leq i<j<h\leq n \\ \alpha_{ij}\alpha_{jh}\alpha_{hi} \neq 1}}\calV(x_ix_jx_h) \quad \subset \PP^{n-1}. \]
In particular, it is the union of a collection of subspaces $\PP(i_1,\dots,i_s)$.
\end{thm}

\begin{rem}\label{rem:pv}
By Theorem \ref{thm.ps}, the point scheme $\Gamma_{\alpha}$ is reduced.
Therefore, we refer to $\Gamma_{\alpha}$ as the \emph{point variety of $S_{\alpha}$} from now on.
\end{rem}

\begin{ex}
\begin{enumerate}
\item The point variety of a standard graded skew polynomial algebra in $1$ variable is given by $\PP^0=\PP(1)$.
\item The point variety of a standard graded skew polynomial algebra in $2$ variables is given by $\PP^1= \PP(1,2)$.
\item The point variety of a standard graded skew polynomial algebra in $3$ variables is isomorphic to $\PP^2=\PP(1,2,3)$ or $\PP(2,3) \cup \PP(1,3) \cup \PP(1,2)$.
\end{enumerate}
\end{ex}

The classification of the point varieties of standard graded skew polynomial algebras in $4$ variables is as follows.

\begin{prop}[{\cite[Corollary 5.1]{Vi}, \cite[Section 4.2]{BDL}}] \label{prop.4ps}
Let $S_{\alpha}$ be a graded skew polynomial algebra in $4$ variables.
Then the point variety of $S_\alpha$ is isomorphic to one of the following:
\begin{enumerate}
\item[(\rnum 1)] $\PP^3=\PP(1,2,3,4)$,
\item[(\rnum 2)] $\PP(1,2,4) \cup \PP(1,2,3) \cup \PP(3,4)$,
\item[(\rnum 3)] $\PP(2,3,4) \cup \PP(1,4) \cup \PP(1,3) \cup \PP(1,2)$,
\item[(\rnum 4)] $\PP(3,4) \cup \PP(2,4) \cup \PP(2,3) \cup \PP(1,4) \cup \PP(1,3) \cup \PP(1,2)$.
\end{enumerate}
\end{prop}

\begin{rem} \label{rem:psd}
Some of the point varieties of skew polynomial algebras do not appear as the point varieties of  ($\pm 1$)-skew polynomial algebras.
For example, (\rnum 3) in Proposition \ref{prop.4ps} does not appear as the point variety of  a ($\pm 1$)-skew polynomial algebra (see Example \ref{ex:n=4}). 
Let $\#\Iso \Gamma_{\alpha}^n$ (resp.\ $\#\Iso \Gamma_{\e}^n$) be the number of the isomorphism classes of the point varieties of skew polynomial algebras (resp.\ ($\pm 1$)-skew  polynomial algebras) in $n$ variables. Then we know
\[
\renewcommand{\arraystretch}{1.2}
\begin{array} {c|cccccc}  
n &1&2&3&4&5&\cdots \\ \hline
\#\Iso \Gamma_{\alpha}^n&1&1&2&4&16&\cdots\\
\#\Iso \Gamma_{\e}^n&1&1&2&3&7&\cdots
\end{array}
\]
(see \cite[Section 4]{BDL}, \cite[Section 6.4]{MU}; see also Examples \ref{ex:n=3}, \ref{ex:n=4}, \ref{ex:n=5}).
\end{rem}

In \cite{Mss}, Mori studied the relationship between noncommutative projective schemes and point schemes in a more general setting.
The following theorem plays a key role in our main result.
\begin{thm} \label{thm:mo}
Let $S_{\alpha}$ and $S_{\alpha'}$ be standard graded skew polynomial algebras.
If $\Projn S_{\alpha}$ and $\Projn S_{\alpha'}$ are isomorphic, then 
the point varieties $\Gamma_{\alpha}$ and $\Gamma_{\alpha'}$ are isomorphic.
\end{thm}

\begin{proof}
This is a special case of \cite[Theorem 3.4]{Mss}.
\end{proof}

We now prepare a lemma that we will use later.
For $X=\{i_1,\dots,i_s\} \subset [n]$ and a permutation $\sigma \in \mathfrak{S}_n$,
let $\sigma (X) = \{\sigma(i_1),\dots,\sigma(i_s)\}$.

\begin{lem}\label{lem:psiso}
Let $S_{\alpha}$ and $S_{\alpha'}$ be standard graded skew polynomial algebras of $\dim S_\alpha=n$ and $\dim S_{\alpha'}=n'$.
If the point varieties $\Gamma_\alpha= \PP(X_1) \cup \cdots \cup \PP(X_r)$ and $\Gamma_{\alpha'}$ are isomorphic,
then $n=n'$ and there exists a permutation $\sigma \in \mathfrak{S}_n$ such that
$\Gamma_{\alpha'}=\PP(\sigma(X_1)) \cup \cdots \cup \PP(\sigma(X_r))$.
\end{lem}

\begin{proof}
Let $\phi: \Gamma_\alpha \to \Gamma_{\alpha'}$ be an isomorphism.
Since an irreducible component of $\Gamma_{\alpha'}$ is of the form $\PP(Y)$ for some $Y \subset [n']$, we have $\Gamma_{\alpha'} = \PP(Y_1) \cup \cdots \cup \PP(Y_r)$, where
$\PP(Y_i)=\phi(\PP(X_i))$ and $Y_i$ is a subset of $[n']$ such that $|Y_i|=|X_i|$ for every $1\leq i\leq r$.
For every nonempty subset $I \subset [r]$, $\phi$ restricts to an isomorphism 
$\phi_{I}:\PP\left(\bigcap _{i \in I}X_i\right)= \bigcap _{i \in I}\PP\left(X_i\right) \to \bigcap _{i \in I}\PP\left(Y_i\right)= \PP\left(\bigcap _{i \in I}Y_i\right)$,
so we see that $|\bigcap _{i \in I}X_i|=|\bigcap _{i \in I}Y_i|$.
Thus there exists a (not necessarily unique) bijection $\sigma: \bigcup_{1 \leq i \leq r} X_i \to \bigcup_{1 \leq i \leq r} Y_i$ such that
$\sigma(\bigcap _{i \in I}X_i)=\bigcap _{i \in I}Y_i$ for every $\emptyset \neq I \subset [r]$.
In particular, $\sigma(X_i)=Y_i$ for every $1 \leq i \leq r$.
Moreover, since $\PP(i,j)$ are contained in $\Gamma_{\alpha}$ for all $1\leq i<j \leq n$, we have $\bigcup_{1 \leq i \leq r} X_i =[n]$. Similarly, $\bigcup_{1 \leq i \leq r} Y_i =[n']$.
Therefore, we obtain $n=n'$ and $\sigma \in \mathfrak{S}_n$.
\end{proof}

\section{Combinatorial aspects of $(\pm 1)$-skew polynomial algebras}
In this section, we discuss combinatorial aspects of $(\pm 1)$-skew polynomial algebras $S_{\e} = k\langle x_1, \dots, x_n \rangle /(x_ix_j -\e_{ij} x_jx_i)$ (defined in Definition \ref{dfn:pm1}).

A \emph{graph} $G$ consists of a set of vertices $V(G)$ and a set of edges $E(G)$ between two vertices. 
In this paper, we always assume that $V(G)$ is a finite set and $G$ has neither loops nor multiple edges. 
An edge between two vertices $v, w\in V(G)$ is written by $vw \in E(G)$. 

\begin{dfn} \label{dfn:g}
Let $S_\e = k\langle x_1, \dots, x_n \rangle /(x_ix_j -\e_{ij} x_jx_i)$ be a $(\pm 1)$-skew polynomial algebra.
The \emph{graph $G_{\e}$ associated to $S_\e$} is defined to be the graph with the vertex set 
$V(G_{\e})=\{1, \dots , n\}=[n]$ and the edge set $E(G_{\e})=\{ij \mid \e_{ij}=\e_{ji}=1, i \neq j \}$. 
\end{dfn}

We say that two graphs $G$ and $G'$ are \emph{isomorphic} as graphs if there exists a bijection $f$ between $V(G)$ and $V(G')$ which induces a bijection between $E(G)$ and $E(G')$. 
\begin{prop} \label{prop:giso}
Two standard graded $(\pm 1)$-skew polynomial algebras $S_{\e}$ and $S_{\e'}$ are isomorphic as graded algebras if and only if $G_{\e}$ and $G_{\e'}$ are isomorphic as graphs. 
\end{prop}

\begin{proof}
This follows from Proposition \ref{prop:skewiso}.
\end{proof}

First, let us recall the notion of mutation of graphs.

\begin{dfn}[{\cite[Definition 6.3]{MU}}]\label{dfn:mu}
Let $G$ be a graph and $v \in V(G)$ a vertex. 
The \emph{mutation $\mu_v(G)$ of $G$ at $v$} is the graph $\mu_v(G)$ defined by $V(\mu_v(G))=V(G)$ and
$$E(\mu_v(G))=\{vw\mid vw \not \in E(G), w\neq v\}\cup \{uw \mid uw\in E(G), u\neq v, w\neq v\}.$$
We say that two graphs $G$ and $G'$ are \emph{mutation equivalent} if $G'$ can be obtained (up to isomorphism) from $G$ by a sequence of mutations at some vertices.
\end{dfn}

Note that the notion of mutation was recently introduced in \cite{MU}
to compute the stable categories of graded maximal Cohen-Macaulay modules over certain noncommutative quadric hypersurfaces (see also \cite{HU}).
However, the same notion, called switching, was already introduced by van~Lint and Seidel \cite{vLS} and has been well-studied in the context of algebraic graph theory (see \cite[Section 11]{GR}).  
For later use, we also recall the notion of switching of graphs. 

\begin{dfn}[{\cite[Section 11.5]{GR}}]\label{dfn:sw}
Let $G$ be a graph and let ${\scS} \subset V(G)$. 
We define the graph $G^{\scS}$ on the vertex set $V(G^{\scS})=V(G)$ with the edge set 
\begin{align*}
E(G^{\scS})=
&\ (\{uv \mid u \in {\scS} \text{ and }v \not\in {\scS}, \text{ or }u \not\in {\scS} \text{ and }v \in {\scS}\} \setminus E(G)) \\
&\ \cup (\{uv \mid u \in {\scS} \text{ and }v \in {\scS}, \text{ or }u \not\in {\scS} \text{ and }v \not\in {\scS}\} \cap E(G)). 
\end{align*}
That is, $G^{\scS}$ is obtained by changing all the edges between ${\scS}$ and $V(G) \setminus {\scS}$ to nonedges, 
and all nonedges between ${\scS}$ and $V(G) \setminus {\scS}$ to edges. 
This operation is called \emph{switching of $G$} on ${\scS}$.
We say that two graphs $G$ and $G'$ are \emph{switching equivalent} if $G'$ is isomorphic to $G^{\scS}$ for some ${\scS} \subset V(G)$.
\end{dfn}

\begin{rem}
\begin{enumerate}
\item $G^{\scS}=G^{V(G) \setminus {\scS}}$ holds.
\item For ${\scS}=\{v_1,\ldots,v_m\} \subset V(G)$, we see that $$G^{\scS}=\mu_{v_m}(\cdots\mu_{v_2}(\mu_{v_1}(G))\cdots).$$
Note that this is independent of the order of $v_i$'s. 
Therefore, mutation equivalence and switching equivalence are the same.
\end{enumerate}
\end{rem}

\begin{ex}\label{ex.mu}
\begin{enumerate}
\item
\[G= \xy /r2pc/: 
{\xypolygon5{~={90}~*{\xypolynode}~>{}}},
"1";"3"**@{-},
"1";"5"**@{-},
"2";"3"**@{-},
"3";"4"**@{-},
"3";"5"**@{-},
\endxy
\quad \Longrightarrow  \quad
\mu_{1}(G) = \xy /r2pc/: 
{\xypolygon5{~={90}~*{\xypolynode}~>{}}},
"1";"2"**@{-},
"1";"4"**@{-},
"2";"3"**@{-},
"3";"4"**@{-},
"3";"5"**@{-},
\endxy
\]
\item
\[G= \xy /r2pc/: 
{\xypolygon5{~={90}~*{\xypolynode}~>{}}},
"1";"2"**@{-},
"1";"5"**@{-},
"2";"3"**@{-},
"2";"4"**@{-},
"5";"3"**@{-},
"5";"4"**@{-},
\endxy
\quad \Longrightarrow  \quad
G^{\{3,4\}}=
\mu_{4}(\mu_{3}(G))= \mu_{3}(\mu_{4}(G))= \xy /r2pc/: 
{\xypolygon5{~={90}~*{\xypolynode}~>{}}},
"1";"2"**@{-},
"1";"3"**@{-},
"1";"4"**@{-},
"1";"5"**@{-},
\endxy
\]
\item
\[G= \xy /r2pc/: 
{\xypolygon6{~={90}~*{\xypolynode}~>{}}},
"1";"2"**@{-},
"1";"4"**@{-},
"1";"6"**@{-},
"6";"5"**@{-},
"6";"3"**@{-},
"2";"3"**@{-},
"2";"5"**@{-},
"3";"4"**@{-},
"5";"4"**@{-},
\endxy
\quad \Longrightarrow  \quad
G^{\{1,2,3\}}=G^{\{4,5,6\}}=
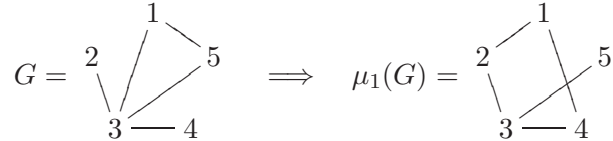
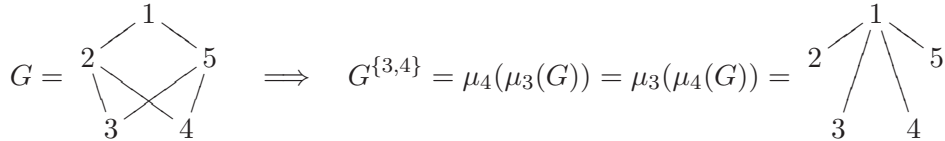
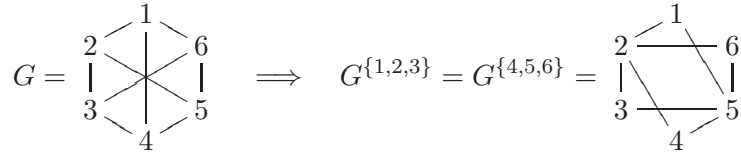
\xy /r2pc/: 
{\xypolygon6{~={90}~*{\xypolynode}~>{}}},
"1";"2"**@{-},
"2";"4"**@{-},
"4";"5"**@{-},
"1";"5"**@{-},
"2";"3"**@{-},
"3";"5"**@{-},
"5";"6"**@{-},
"6";"2"**@{-},
\endxy
\]
\end{enumerate}
\end{ex}

For a given graph $G$, we can define a new graph $\Sw(G)$, called the switching graph of $G$.

\begin{dfn}[{\cite[Section 11.5]{GR}}]\label{dfn:swg}
Let $G$ be a graph.  Then the \emph{switching graph of $G$}, denoted by $\Sw(G)$, is the graph on the vertex set $V(\Sw(G))=V(G) \times \{0,1\}$ with edges defined as follows:
\begin{itemize}
\item $(u,0)$ and $(v,0)$ are adjacent if and only if $uv \in E(G)$, 
\item $(u,1)$ and $(v,1)$ are adjacent if and only if $uv \in E(G)$, 
\item $(u,0)$ and $(v,1)$ are adjacent if and only if $uv \not\in E(G)$. 
\end{itemize}
\end{dfn}

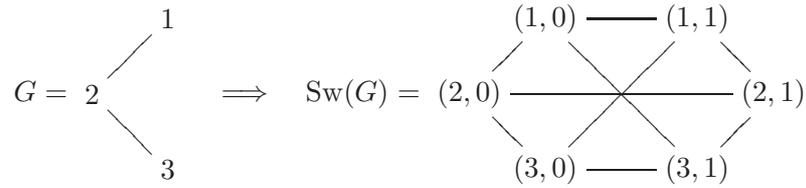
\begin{ex}
\[G=
\xygraph{
!{<0cm,0cm>;<1cm,0cm>:<0cm,1cm>::}
!{(0,0) }*+{2}="2"
!{(1,1) }*+{1}="1"
!{(1,-1) }*+{3}="3"
"1"-"2" "2"-"3"
} 
\quad \Longrightarrow  \quad
\Sw(G)=
\xygraph{
!{<0cm,0cm>;<1cm,0cm>:<0cm,1cm>::}
!{(0,0) }*+{(2,0)}="2.0"
!{(1,1) }*+{(1,0)}="1.0"
!{(1,-1) }*+{(3,0)}="3.0"
!{(4,0) }*+{(2,1)}="2.1"
!{(3,1) }*+{(1,1)}="1.1"
!{(3,-1) }*+{(3,1)}="3.1"
"1.0"-"2.0" "2.0"-"3.0"
"1.1"-"2.1" "2.1"-"3.1"
"1.0"-"1.1" "2.0"-"2.1" "3.0"-"3.1"
"1.0"-"3.1" "1.1"-"3.0"
} 
\]
\end{ex}

The following characterization of switching equivalence is known.

\begin{prop}[{See \cite[Section 11.5]{GR}}]\label{prop:switch}
Two graphs $G$ and $G'$ are switching equivalent if and only if $\Sw(G)$ and $\Sw(G')$ are isomorphic as graphs. 
\end{prop}

Classifying graphs up to switching equivalence is closely related to classifying $(\pm 1)$-skew polynomial algebras, as follows.

\begin{prop}[{\cite[Lemma 6.5]{MU}}]\label{prop:mu}
Let $S_{\e}$ and $S_{\e'}$ be standard graded $(\pm 1)$-skew polynomial algebras.
If the graphs $G_{\e}$ and $G_{\e'}$ are switching equivalent, then 
the categories $\GrMod S_{\e}$ and $ \GrMod S_{\e'}$ are equivalent.
\end{prop}

Next, let us recall what an abstract simplicial complex is and some related notions.
We say that a collection $\Delta$ of subsets of a finite set $V$ is an (abstract) \emph{simplicial complex} on the vertex set $V$ 
if the following conditions are satisfied:
\begin{itemize}
\item $\{v\} \in \Delta$ for each $v \in V$,
\item $F \in \Delta$ and $F' \subset F$ imply $F' \in \Delta$. 
\end{itemize}
The \emph{dimension of $F \in \Delta$}, denoted by $\dim F$, is defined to be $|F|-1$ and we define $\dim \Delta=\max\{\dim F \mid F \in \Delta\}$. 
Let $\calF(\Delta)$ denote the set of all facets (i.e. maximal faces) of $\Delta$. 
For two simplicial complexes $\Delta$ and $\Delta'$ on the same vertex set $V$, 
we say that $\Delta$ and $\Delta'$ are \emph{isomorphic} as simplicial complexes if there exists a permutation $\sigma$ on $V$ 
which induces a bijection between $\calF(\Delta)$ and $\calF(\Delta')$. 

Then, we define the simplicial complex associated to a standard graded $(\pm 1)$-skew polynomial algebra $S_{\e}$.
Let $G_\e$ be the graph associated to $S_{\e}$ as in Definition \ref{dfn:g}.

\begin{dfn}\label{dfn:Del}
We define the simplicial complex $\Delta_\e$ on $V(G_\e)=[n]$ as follows. 
A subset $F \subset V( G_\e) $ belongs to $\Delta_\e$ if and only if 
for any $\{a,b,c\} \subset F$ where $a,b,c$ are all distinct, one of the following two conditions is satisfied: 
\begin{equation}\label{eq:Del}
\begin{split}
&ab \in E(G_\e), \; ac \in E(G_\e)\text{ and } bc \in E(G_\e), \text{ or}\\
&ab \in E(G_\e), \; ac \not\in E(G_\e)\text{ and }bc \not\in E(G_\e) 
\text{ after replacing $a,b,c$ if necessary.} 
\end{split}
\end{equation}
\end{dfn}

On the other hand, we define the ideal $I_{\e}$ of the commutative polynomial ring $k[x_1,\ldots,x_n]$ as follows:
\begin{align*}
I_{\e}=
&\ (x_ax_bx_c \mid ab \not\in E(G_\e) \text{ and }ac \not\in E(G_\e) \text{ and } bc \not\in E(G_\e)) \\
&\ +(x_ax_bx_c \mid ab \in E(G_\e) \text{ and }ac \in E(G_\e) \text{ and } bc \not\in E(G_\e)). 
\end{align*}
Then it follows from Theorem \ref{thm.ps} that $\Gamma_\e=\calV(I_{\e})$.
Since $I_\e$ is a square-free monomial ideal, we can associate the simplicial complex whose Stanley-Reisner ideal coincides with $I_{\e}$. 
We see that such simplicial complex is nothing but $\Delta_\e$ defined above. 
(For the fundamental materials on Stanley-Reisner ideals, please consult, e.g., \cite[Section 5]{BH}.) 

Note that the structure of the point variety $\Gamma_\e$ can be read from the combinatorial structure of $\Delta_\e$. 
In fact, the primary decomposition of $I_{\e}$ corresponds to the irreducible decomposition of $\Gamma_\e$ 
and the primary decomposition of a Stanley-Reisner ideal $I_{\e}$ can be obtained from the structure of the facets of $\Delta_\e$. 
More precisely, we have 
\[
I_{\e} = \bigcap_{F \in \calF(\Delta_\e)}P_F,
\]
where $P_F=( x_i \mid i \not\in F)$ is a prime ideal of $k[x_1,\ldots,x_n]$ associated to $F \in \Delta_\e$ 
(see \cite[Theorem 5.1.4]{BH}). Hence, we see that 
\begin{align}\label{eq:ps}
\Gamma_\e=\bigcup_{F \in \calF(\Delta_\e)}\PP(F),
\end{align}
that is,
\begin{align}\label{eq:calFD}
\calF(\Delta_\e)=\{F \subset [n] \mid \PP(F) \ \text{is an irreducible component of $\Gamma_\e$}\}.
\end{align}
In particular, we see that the dimension of the point variety $\Gamma_\e$ coincides with the dimension of the simplicial complex $\Delta_\e$. 

Now we are in the position to prove the following theorem.
\begin{thm}\label{thm:Deltom}
Let $S_{\e}$ and $S_{\e'}$ be standard graded $(\pm 1)$-skew polynomial algebras.
If $\Delta_\e$ and $\Delta_{\e'}$ are isomorphic as simplicial complexes, 
then $G_\e$ and $G_{\e'}$ are switching equivalent.
\end{thm}

\begin{proof}
For simplicity, we write $G=G_{\e}$ and $G'=G_{\e'}$.
Let $f:V(G) \rightarrow V(G')$ be a bijection which induces an isomorphism between $\Delta_\e$ and $\Delta_{\e'}$. 

Fix a vertex $u \in V(G)$. Let 
\begin{align*}
{\scS}=&\ \{ v \in V(G) \mid uv \in E(G) \text{ and }f(u)f(v) \not\in E(G')\} \\
&\ \cup \{ v \in V(G) \mid uv \not\in E(G) \text{ and }f(u)f(v) \in E(G')\}.
\end{align*}
Note that ${\scS}$ depends on the choice of $u$, and that $u \not\in {\scS}$. 
In what follows, we prove that $f$ induces an isomorphism between $G^{\scS}$ and $G'$ as graphs, i.e., $G$ and $G'$ are switching equivalent. 
To do so, what we have to do is to show that $ab \in E(G^{\scS})$ if and only if $f(a)f(b) \in E(G')$. 

Let $v \in V(G) \setminus \{u\}$. 
\begin{itemize}
\item Assume that $uv \in E(G)$. 
\begin{itemize}
\item If $v \in {\scS}$, then $uv \not\in E(G^{\scS})$ since $u \in V(G)\setminus {\scS}$, $v \in {\scS}$ and $uv \in E(G)$. 
On the other hand, by definition of ${\scS}$, we have $f(u)f(v) \not\in E(G')$. 
\item If $v \not\in {\scS}$, then $uv \in E(G^{\scS})$. On the other hand, by definition of ${\scS}$, we have $f(u)f(v) \in E(G')$. 
\end{itemize}
\item Assume that $uv \not\in E(G)$. 
\begin{itemize}
\item If $v \in {\scS}$, then $uv \in E(G^{\scS})$. On the other hand, we have $f(u)f(v) \in E(G')$. 
\item If $v \not\in {\scS}$, then $uv \not\in E(G^{\scS})$. On the other hand, we have $f(u)f(v) \not\in E(G')$. 
\end{itemize}
\end{itemize}

Let $v,w \in V(G) \setminus \{u\}$. 
\begin{itemize}
\item Consider the case $v,w \in {\scS}$. Then we see that $uv \in E(G)$ if and only if $f(u)f(v) \not\in E(G')$, and $uw \in E(G)$ if and only if $f(u)f(w) \not\in E(G')$.
Furthermore, we see that $vw \in E(G)$ if and only if $vw \in E(G^{\scS})$.

\begin{itemize}
\item Assume that $uv \in E(G)$ and $uw \in E(G)$. 

When $vw \in E(G)$, we have $\{u,v,w\} \in \Delta_\e$ by \eqref{eq:Del}. Hence, $\{f(u),f(v),f(w)\} \in \Delta_{\e'}$. 
Since $f(u)f(v) \not\in E(G')$ and $f(u)f(w) \not\in E(G')$ by $v,w \in {\scS}$, we conclude that $f(v)f(w) \in E(G')$ by \eqref{eq:Del}. 

When $vw \not\in E(G)$, we have $\{u,v,w\} \not\in \Delta_\e$ by \eqref{eq:Del}. Hence, $\{f(u),f(v),f(w)\} \not\in \Delta_{\e'}$. 
Since $f(u)f(v) \not\in E(G')$ and $f(u)f(w) \not\in E(G')$, we conclude that $f(v)f(w) \not\in E(G')$ by \eqref{eq:Del}. 

\item Assume that $uv \not\in E(G)$ and $uw \not\in E(G)$. 

When $vw \in E(G)$, we have $\{u,v,w\} \in \Delta_\e$ by \eqref{eq:Del}. Hence, $\{f(u),f(v),f(w)\} \in \Delta_{\e'}$. 
Since $f(u)f(v) \in E(G')$ and $f(u)f(w) \in E(G')$ by $v,w \in {\scS}$, we conclude that $f(v)f(w) \in E(G')$ by \eqref{eq:Del}. 

When $vw \not\in E(G)$, we have $\{u,v,w\} \not\in \Delta_\e$ by \eqref{eq:Del}. Hence, $\{f(u),f(v),f(w)\} \not\in \Delta_{\e'}$. 
Since $f(u)f(v) \in E(G')$ and $f(u)f(w) \in E(G')$, we conclude that $f(v)f(w) \not\in E(G')$ by \eqref{eq:Del}. 
\item Assume that $uv \in E(G)$ and $uw \not\in E(G)$. 

When $vw \in E(G)$, we have $\{u,v,w\} \not\in \Delta_\e$ by \eqref{eq:Del}. Hence, $\{f(u),f(v),f(w)\} \not\in \Delta_{\e'}$. 
Since $f(u)f(v) \not\in E(G')$ and $f(u)f(w) \in E(G')$ by $v,w \in {\scS}$, we conclude that $f(v)f(w) \in E(G')$ by \eqref{eq:Del}. 

When $vw \not\in E(G)$, we have $\{u,v,w\} \in \Delta_\e$ by \eqref{eq:Del}. Hence, $\{f(u),f(v),f(w)\} \in \Delta_{\e'}$. 
Since $f(u)f(v) \not\in E(G')$ and $f(u)f(w) \in E(G')$, we conclude that $f(v)f(w) \not\in E(G')$ by \eqref{eq:Del}. 
\end{itemize}
\item Consider the case $v,w \in V(G) \setminus {\scS}$. Then similar arguments to the above can be applied.
As a consequence, we obtain that the adjacency of $v$ and $w$ in $E(G^{\scS})$ coincides with the adjacency of $f(v)$ and $f(w)$ in $E(G')$. 
\item Even in the case where $v \in {\scS}$ and $w \in V(G) \setminus {\scS}$, by a routine work as above, we obtain the same conclusion. 
\end{itemize}

Therefore, we conclude that $ab \in E(G^{\scS})$ if and only if $f(a)f(b) \in E(G')$ for any $a,b \in V(G)$. 
\end{proof}

\section{Main Results}

In this section, we present the proof of the main theorem (Theorem \ref{thm:main}).
To do so, we first give the definition of $\overline{S_\e}$.

\begin{dfn}\label{dfn:overline}
Let $S_\e = k\langle x_1, \dots, x_n \rangle /(x_ix_j -\e_{ij} x_jx_i)$ be a $(\pm 1)$-skew polynomial algebra in $n$ variables.
From $S_\e$, we define a new standard graded $(\pm 1)$-skew polynomial algebra $\overline{S_\e}$ in $2n$ variables to be
the quotient of $k\langle x_1, \dots, x_n, y_1, \dots, y_n \rangle$ by the ideal generated by
\begin{align*}
&x_ix_j -\e_{ij} x_jx_i \quad  (1\leq i,j \leq n),\\
&y_iy_j -\e_{ij} y_jy_i \quad  (1\leq i,j \leq n),\\
&x_iy_j +\e_{ij} y_jx_i \quad  (1\leq i,j \leq n, i\neq j),\\
&x_iy_i - y_ix_i \quad  (1\leq i \leq n).
\end{align*}
\end{dfn}
It is easy to see that the graph associated to $\overline{S_\e}$ is $\Sw(G_\e)$.

Now we are ready to give the proof of the main theorem.

\begin{thm}\label{thm:main}
Let $S_{\e}$ and $S_{\e'}$ be standard graded $(\pm 1)$-skew polynomial algebras.
Then the following are equivalent.
\begin{enumerate}
\item $G_\e$ and $G_{\e'}$ are switching (or mutation) equivalent.
\item $\Sw(G_\e)$ and $\Sw(G_{\e'})$ are isomorphic as graphs.
\item $\overline{S_\e}$ and $\overline{S_{\e'}}$ are isomorphic as graded algebras.
\item $\GrMod S_\e$ and $\GrMod S_{\e'}$ are equivalent as categories.
\item $\Projn S_\e$ and $\Projn S_{\e'}$ are isomorphic as noncommutative projective schemes.
\item $\Gamma_\e$ and $\Gamma_{\e'}$ are isomorphic as varieties.
\item $\Delta_\e$ and $\Delta_{\e'}$ are isomorphic as simplicial complexes.
\end{enumerate}
\end{thm}

\begin{proof}
We check the following implications:
\[
\xymatrix@R=2pc@C=2pc{
(3) \ar@{<=>}[r]&(2)  \ar@{<=>}[r]&(1)  \ar@{=>}[r] &(4) \ar@{=>}[r]&(5) \ar@{=>}[r] &(6) \ar@{=>}[r] &(7) \ar@/^10pt/@<1ex>@{=>}[llll]
}\vspace{4mm}
\]

$(1) \Rightarrow (4):$ This is true by Proposition \ref{prop:mu}.

$(4) \Rightarrow (5):$ This follows from Theorem \ref{thm:z}.

$(5) \Rightarrow (6):$ This follows from Theorem \ref{thm:mo}.

$(6) \Rightarrow (7):$ By Lemma~\ref{lem:psiso}, there exists a permutation $\sigma \in \mathfrak{S}_n$ which induces a bijection between the coordinates of the irreducible components of $\Gamma_\e$ and $\Gamma_{\e'}$. By \eqref{eq:calFD}, we see that this $\sigma$ induces a bijection between $\calF(\Delta_\e)$ and $\calF(\Delta_{\e'})$. 

$(7) \Rightarrow (1):$ This implication is the content of Theorem \ref{thm:Deltom}.

$(1) \Leftrightarrow (2):$ This is just Proposition \ref{prop:switch}.

$(2) \Leftrightarrow (3):$ Since the graphs associated to $\overline{S_\e}$ and $\overline{S_{\e'}}$ are $\Sw(G_\e)$ and $\Sw(G_{\e'})$, respectively, the assertion follows from Proposition \ref{prop:giso}.
\end{proof}

In particular, by $(1) \Leftrightarrow (5)$ in Theorem \ref{thm:main}, we can completely classify $(\pm 1)$-skew projective spaces by a purely combinatorial method.

\begin{ex}[{\cite[Example 6.14]{MU}}] \label{ex:n=3}
Consider the case $n=3$. There exist two $(\pm 1)$-skew projective spaces $\Projn S_\e$ of dimension $2$ up to isomorphism.
In fact, there exist two graphs
\[\text{(\rnum 1)}\;\; \xy /r2pc/: 
{\xypolygon3{~={90}~*{\xypolynode}~>{}}},
"1";"2"**@{-},
\endxy
\qquad\qquad
\text{(\rnum 2)}\;\; \xy /r2pc/: 
{\xypolygon3{~={90}~*{\xypolynode}~>{}}},
\endxy
\qquad
\]
up to switching equivalence. The corresponding point varieties are
\begin{enumerate}
\item[(\rnum 1)] $\PP^2=\PP(1,2,3)$,
\item[(\rnum 2)] $\PP(2,3) \cup \PP(1,3) \cup \PP(1,2)$.
\end{enumerate}
\end{ex}

\begin{ex}[{\cite[Section 6.4.1]{MU}}]\label{ex:n=4}
Consider the case $n=4$. There exist three $(\pm 1)$-skew projective spaces $\Projn S_\e$ of dimension $3$ up to isomorphism.
In fact, there exist three graphs
\[\text{(\rnum 1)}\; \xy /r2pc/: 
{\xypolygon4{~={90}~*{\xypolynode}~>{}}},
"1";"2"**@{-},
"3";"4"**@{-},
\endxy
\qquad\quad
\text{(\rnum 2)}\; \xy /r2pc/: 
{\xypolygon4{~={90}~*{\xypolynode}~>{}}},
"1";"2"**@{-},
\endxy
\qquad\quad
\text{(\rnum 3)}\; \xy /r2pc/: 
{\xypolygon4{~={90}~*{\xypolynode}~>{}}},
\endxy
\]
up to switching equivalence. The corresponding point varieties are
\begin{enumerate}
\item[(\rnum 1)] $\PP^3=\PP(1,2,3,4)$,
\item[(\rnum 2)] $\PP(1,2,4) \cup \PP(1,2,3) \cup \PP(3,4)$,
\item[(\rnum 3)] $\PP(3,4) \cup \PP(2,4) \cup \PP(2,3) \cup \PP(1,4) \cup \PP(1,3) \cup \PP(1,2)$.
\end{enumerate}
\end{ex}

\begin{ex}[{\cite[Section 6.4.2]{MU}}]\label{ex:n=5}
Consider the case $n=5$. There exist seven $(\pm 1)$-skew projective spaces $\Projn S_\e$ of dimension $4$ up to isomorphism.
In fact, there exist seven graphs
\begin{center}
(\rnum 1) $\xy /r2pc/: 
{\xypolygon5{~={90}~*{\xypolynode}~>{}}},
"1";"2"**@{-},
"2";"3"**@{-},
"3";"1"**@{-},
"4";"5"**@{-},
\endxy $
\qquad 
(\rnum 2) $\xy /r2pc/: 
{\xypolygon5{~={90}~*{\xypolynode}~>{}}},
"1";"2"**@{-},
"2";"3"**@{-},
"4";"5"**@{-},
\endxy $
\qquad
(\rnum 3) $\xy /r2pc/: 
{\xypolygon5{~={90}~*{\xypolynode}~>{}}},
"1";"2"**@{-},
"3";"4"**@{-},
\endxy $
\qquad 
(\rnum 4) $\xy /r2pc/: 
{\xypolygon5{~={90}~*{\xypolynode}~>{}}},
"1";"2"**@{-},
"2";"3"**@{-},
"3";"4"**@{-},
\endxy $

(\rnum 5) $\xy /r2pc/: 
{\xypolygon5{~={90}~*{\xypolynode}~>{}}},
"1";"2"**@{-},
"2";"3"**@{-},
\endxy $
\qquad 
(\rnum 6) $\xy /r2pc/: 
{\xypolygon5{~={90}~*{\xypolynode}~>{}}},
"1";"2"**@{-},
\endxy$
\qquad 
(\rnum 7) $\xy /r2pc/: 
{\xypolygon5{~={90}~*{\xypolynode}~>{}}},
\endxy $
\end{center}
up to switching equivalence. The corresponding point varieties are
\begin{enumerate}
\item[(\rnum 1)] $\PP^4=\PP(1,2,3,4,5)$,
\item[(\rnum 2)] $\PP(2,3,4,5)\cup \PP(1,2,4,5)\cup \PP(1,3)$,
\item[(\rnum 3)] $\PP(1,2,3,4)\cup \PP(3,4,5)\cup \PP(1,2,5)$,
\item[(\rnum 4)] $\PP(3,4,5)\cup \PP(2,3,5)\cup \PP(1,3,4)\cup \PP(1,2,5)\cup \PP(1,2,4)$,
\item[(\rnum 5)] $\PP(2,3,5)\cup\PP(2,3,4)\cup \PP(1,2,5)\cup \PP(1,2,4)\cup\PP(4,5)\cup \PP(1,3)$,
\item[(\rnum 6)] $\PP(1,2,5)\cup \PP(1,2,4)\cup \PP(1,2,3)\cup \PP(4,5)\cup \PP(3,5)\cup \PP(3,4)$,
\item[(\rnum 7)] $\PP(4,5)\cup\PP(3,5)\cup\PP(3,4)\cup\PP(2,5)\cup\PP(2,4)\cup\PP(2,3)\cup\PP(1,5)\cup\PP(1,4)\cup\PP(1,3)\cup\PP(1,2)$.
\end{enumerate}
\end{ex}

\begin{rem}\label{rem:vi}
Note that the equivalence $(4) \Leftrightarrow (5)$ in Theorem \ref{thm:main} shows that
Vitoria's conjecture (Conjecture \ref{conj:vi}) holds for $(\pm 1)$-skew polynomial algebras.
\end{rem}

\begin{rem}\label{rem:gsp}
We remark that it follows from \cite[Example 4.10]{Mss} that $(6) \Rightarrow (4)$ in Theorem \ref{thm:main} does not hold for general skew polynomial algebras.
\end{rem}

\section{Some invariants}
In this section, we study some invariants for the classification obtained in Theorem \ref{thm:main}.
In particular, we discuss information about the irreducible components of the point variety $\Gamma_{\e}$ via the simplicial complex $\Delta_\e$ and the graph $G_\e$.
By \eqref{eq:ps}, the face structure of $\Delta_\e$ is crucial for analyzing of $\Gamma_\e$. 
To study the faces of $\Delta_\e$ using $G_\e$, we recall the notion of cliques of graphs.

Let $G$ be a graph. We say that $C \subset V(G)$ is a \emph{clique of $G$} if any two vertices in $C$ are adjacent. 
Note that the empty set $\emptyset$ is regarded as a clique of $G$. 
We call the maximum of the cardinalities of all cliques of $G$ the \emph{clique number of $G$}, and we denote it by $\omega(G)$.

We now discuss the face structure of $\Delta_\e$ in terms of cliques of $G_\e$ and $\Sw(G_\e)$. 
Let $\pi:V(\Sw(G_\e)) \rightarrow V(G_\e)$ be a map $\pi((u,0))=\pi((u,1))=u$ for each $u \in V(G_\e)$. 

\begin{thm}\label{thm:Delta}
Let $S_{\e}$ be a standard graded $(\pm 1)$-skew polynomial algebra in $n$ variables.
For $F \subset V(G_\e)=[n]$, the following conditions are equivalent.
\begin{enumerate}
\item $F \in \Delta_\e$. 
\item $F$ consists of two disjoint cliques that are not connected by an edge, 
i.e., $F=F' \sqcup F''$ where $F'$ and $F''$ are cliques of $G_\e$ such that there is no edge between $F'$ and $F''$.
\item $F=\pi(C)$ for some clique $C$ of $\Sw(G_\e)$. 
\end{enumerate}
\end{thm}

\begin{proof}
$(1) \Rightarrow (2)$: Take $F \in \Delta_\e$. 
Note that any subset $F$ of $V(G_\e)$ with $|F|=1$ or $2$ is a face of $\Delta_\e$ by \eqref{eq:Del}. 
Thus, we may assume that $|F| \geq 3$.
By \eqref{eq:Del}, we see that there is at least one edge among the vertices in $F$. 
Take a maximal clique $F' \subset F$ of $G_\e$ contained in $F$. Then $|F'| \geq 2$. 
When $F'=F$, we may regard $F$ as $F' \sqcup \emptyset$. 
Assume $F' \subsetneq F$. If there is an edge $uv \in E(G_\e)$ with $u \in F'$ and $v \in F \setminus F'$, 
then there is $u' \in F'$ with $u'v \not\in E(G_\e)$ by the maximality of $F'$, 
so we see that $\{u,u',v\}$ does not satisfy the condition \eqref{eq:Del}, a contradiction. 
Hence, there is no edge between $F'$ and $F \setminus F'$. 
Similarly, if $F \setminus F'$ is not a clique, then we can lead a contradiction.
Therefore, we conclude the assertion (2). 

$(2) \Rightarrow (1)$: Let $F$ be a union of two disjoint cliques that are not connected by an edge. 
Then it is straightforward to check that $F$ satisfies that any $\{a, b, c\} \subset F$ satisfies \eqref{eq:Del}. 

$(2) \Rightarrow (3)$: 
Let $F=F' \sqcup F''$ satisfying the condition in (2).
Since there is no edge between $F'$ and $F''$, 
we see that $C=\{(u,0) \mid u \in F'\} \cup \{(v,1) \mid v \in F''\}$ is a clique of $\Sw(G_\e)$ and $F=\pi(C)$. 

$(3) \Rightarrow (2)$:
Let $C=C' \sqcup C''$ with $C' \subset V(G_\e) \times \{0\}$ and $C'' \subset V(G_\e) \times \{1\}$. 
Note that $C'=\emptyset$ or $C''=\emptyset$ might happen. 
Then $C'$ and $C''$ are also cliques of $\Sw(G_\e)$, i.e., $\pi(C')$ and $\pi(C'')$ are cliques of $G_\e$. 
Take $u \in \pi(C') \cap \pi(C'')$ if exists.
Then we see that $(u,0)$ is adjacent to any vertex in $C' \subset V(G_\e) \times \{0\}$ since $u \in \pi(C')$ and $C'$ is a clique. 
Similarly, $(u,1)$ is adjacent to any vertex in $C'' \subset V(G_\e) \times \{1\}$. 
This means that $(u,0) \in C'$ is \textit{not} adjacent to any vertex in $C''$ by definition of $\Sw(G_\e)$. 
This is a contradiction to that $C' \sqcup C''$ is a clique of $\Sw(G_\e)$. Hence, $\pi(C') \cap \pi(C'')=\emptyset$.
Since $C$ is a clique of $\Sw(G_\e)$, it follows that there is no edge between $\pi(C')$ and $\pi(C'')$. 
Therefore, we conclude that $\pi(C)=\pi(C') \sqcup \pi(C'')$ is a required one.
\end{proof}

As corollaries of this theorem, we obtain the following. 
\begin{cor}
Let $F \subset V(G_\e)=[n]$. 
The following conditions are equivalent.
\begin{enumerate}
\item $\PP(F)$ is an irreducible component of $\Gamma_\e$.
\item $F \in \Delta_\e$ is a facet.
\item $F$ consists of two disjoint maximal cliques that are not connected by an edge, 
i.e., $F=F' \sqcup F''$ where $F'$ and $F''$ are maximal cliques of $G_\e$ such that there is no edge between $F'$ and $F''$.
\item $F=\pi(C)$ for some maximal clique $C$ of $\Sw(G_\e)$. 
\end{enumerate}
\end{cor}

\begin{cor} \label{cor:dim} 
We have
\begin{align*}
&\dim \Gamma_\e \\
&=\dim \Delta_\e\\
&=\max\{ |F|-1 \mid F \text{ is a disjoint union of two cliques of $G_\e$ that are not connected by an edge} \} \\
&=\omega(\Sw(G_\e))-1.
\end{align*}
\end{cor}

This corollary describes the relations between some invariants.

Here, we recall the notion of type of $\Gamma_\e$ introduced in \cite{BDL}.

\begin{dfn}[{\cite[Section 4.2]{BDL}}]
Let $\Gamma_\e$ be the point variety of a standard graded $(\pm 1)$-skew polynomial algebra $S_{\e}$ in $n$ variables.
The \emph{type of $\Gamma_\e$}, denoted by $\type \Gamma_\e$, is the vector $(t_{n-1},\dots,t_1) \in \NN^{n-1}$ where
\[
t_i = \#\{ X \subset [n] \mid \PP(X) \ \text{is an irreducible component of $\Gamma_\e$, and}\ |X|-1=i \}.
\]
\end{dfn}

\begin{ex}\label{ex:tn=4}
The types of the point varieties in Example \ref{ex:n=4} are
\begin{enumerate}
\item[(\rnum 1)] $(1,0,0)$,
\item[(\rnum 2)] $(0,2,1)$,
\item[(\rnum 3)] $(0,0,6)$.
\end{enumerate}
\end{ex}

\begin{ex}\label{ex:tn=5}
The types of the point varieties in Example \ref{ex:n=5} are
\begin{enumerate}
\item[(\rnum 1)] $(1,0,0,0)$,
\item[(\rnum 2)] $(0,2,0,1)$,
\item[(\rnum 3)] $(0,1,2,0)$,
\item[(\rnum 4)] $(0,0,5,0)$,
\item[(\rnum 5)] $(0,0,4,2)$,
\item[(\rnum 6)] $(0,0,3,3)$,
\item[(\rnum 7)] $(0,0,0,10)$.
\end{enumerate}
\end{ex}

Type is an invariant that is more precise than dimension. In fact, if $\type \Gamma_\e=(t_{n-1},\dots,t_1)$, then $\dim \Gamma_\e$ is just $\max \{i \mid t_i \neq 0 \}$. Note that the entry $t_1$ of $\type \Gamma_\e$ (i.e., the number of irreducible components of $\Gamma_\e$ that are isomorphic to $\PP^1$) was well-studied in \cite{HU}.

As long as we see Examples \ref{ex:tn=4}, \ref{ex:tn=5}, it is natural to ask the following question: Is $\type \Gamma_\e$ a complete invariant?
That is, if $\type\Gamma_\e=\type\Gamma_{\e'}$, then are $\Gamma_\e$ and $\Gamma_{\e'}$ isomorphic? Note that Belmans, De~Laet and Le~Bruyn \cite[Section 4.3]{BDL} pointed out that type is not a complete invariant for point varieties of skew (not necessarily $(\pm 1)$-skew) polynomial algebras.
From this, our question asks whether type is a complete invariant by restricting to ``$(\pm 1)$-skew'' (see Remark \ref{rem:psd}).

As we will see soon, when the graph $G_\e$ has a special form,
we can obtain a simple formula for calculating $\type \Gamma_\e$
from Theorem \ref{thm:Delta}.
Furthermore, using the formula, we can give an example, showing that our question has a negative answer, that is, type is not a complete invariant even after restricting to ``$(\pm 1)$-skew''.

Let $K_d$ denote the complete graph with $d$ vertices. 
A \emph{disjoint union $G \sqcup G'$ of graphs $G$ and $G'$} is the graph on the vertex set $V(G) \sqcup V(G')$ with the edge set $E(G) \sqcup E(G')$. 
We discuss the case of a disjoint union of complete graphs. 

\begin{cor}\label{cor:type}
Let $S_{\e}$ be a standard graded $(\pm 1)$-skew polynomial algebra in $n$ variables.
If $G_\e= K_{r_1} \sqcup  K_{r_2} \sqcup \cdots \sqcup K_{r_\ell}$, where $r_1 \leq r_2 \leq \cdots \leq r_\ell$, 
then $\type \Gamma_\e=(t_{n-1},\dots,t_1)$, 
where $$t_i=\#\{(a,b) \mid 1 \leq a<b \leq \ell, r_a+r_b-1=i\}.$$
In particular, $\dim \Gamma_\e=r_{\ell-1}+r_\ell-1$.
\end{cor}

\begin{proof}
This follows easily from Theorem \ref{thm:Delta}.
\end{proof}

\begin{lem} \label{lem:kmut}
Let $G= K_{r_1} \sqcup  K_{r_2} \sqcup \cdots \sqcup K_{r_\ell}$ with $\ell \geq 3$ 
and let $G'= K_{s_1} \sqcup  K_{s_2} \sqcup \cdots \sqcup K_{s_m}$ with $m \geq 3$. 
If $G$ is switching equivalent to $G'$, then $G$ and $G'$ are isomorphic as graphs, i.e., 
$\ell=m$ and $\{r_1,\ldots,r_\ell\}=\{s_1,\ldots,s_m\}$ (as multi-sets). 
\end{lem}

\begin{proof}
For simplicity, let us write $G=\bigsqcup_{i=1}^\ell G_i$ (resp.\ $G'=\bigsqcup_{i=1}^m G_i'$) where $G_i$ (resp.\ $G_i'$) is isomorphic to $K_{r_i}$ (resp.\ $K_{s_i}$). 

Assume that $G$ is switching equivalent to $G'$. Then there is a subset ${\scS} \subset V(G)$ such that $G^{\scS}$ is isomorphic to $G'$ as graphs. 
In what follows, we show that ${\scS}=\emptyset$ or ${\scS}=V(G)$. 

Suppose that ${\scS} \neq \emptyset$. Then we may assume that ${\scS} \cap V(G_1) \neq \emptyset$.
\begin{itemize}
\item Suppose that there are at least two indices $i,j > 1$ such that $V(G_i) \cap {\scS} \subsetneq V(G_i)$ and $V(G_j) \cap {\scS} \subsetneq V(G_j)$. We may assume that $i=2$ and $j=3$. 
Then $G^{\scS}$ contains a connected component whose vertex set contains $(V(G_1) \cap {\scS}) \cup (V(G_2) \setminus {\scS}) \cup (V(G_3) \setminus {\scS})$. 
Since there is no edge between $V(G_2) \setminus {\scS}$ and $V(G_3) \setminus {\scS}$ in $G^{\scS}$, $G^{\scS}$ contains a non-complete connected component, a contradiction. 
Hence, there is at most one index $i>1$ with $V(G_i) \cap {\scS} \subsetneq V(G_i)$. We may assume that $i=2$ if such $i$ exists. Then $V(G_3) \cup \cdots \cup V(G_\ell) \subset {\scS}$. 
\item Suppose that $V(G_2) \cap {\scS} \subsetneq V(G_2)$. Then $G^{\scS}$ contains a non-complete connected component 
whose vertex set contains $(V(G_1) \cap {\scS}) \cup (V(G_2) \setminus {\scS}) \cup (V(G_3) \cap {\scS})$, a contradiction. Hence, $V(G_2) \subset {\scS}$. 
\item Suppose that $V(G_1) \cap {\scS} \subsetneq V(G_1)$. Then $G^{\scS}$ contains a non-complete connected component 
whose vertex set contains $(V(G_1) \setminus {\scS}) \cup (V(G_2) \cap {\scS}) \cup (V(G_3) \cap {\scS})$, a contradiction. Hence, $V(G_1) \subset {\scS}$. 
\end{itemize}
Therefore, $V(G_1) \cup \cdots \cup V(G_\ell) \subset {\scS}$, and hence ${\scS}=V(G)$, as required. 
\end{proof}

\begin{ex} \label{ex:cex}
Let $S_{\e}$ and $S_{\e'}$ be standard graded $(\pm 1)$-skew polynomial algebras such that
$G_\e=K_1 \sqcup K_3 \sqcup K_3 \sqcup K_3$ and $G_{\e'}=K_2 \sqcup K_2 \sqcup K_2 \sqcup K_4$. 
Then
\[ \type \Gamma_\e = (0,0,0,0,3,0,3,0,0) = \type \Gamma_{\e'}\]
by Corollary \ref{cor:type}.
On the other hand, $G_\e$ is not switching equivalent to $G_{\e'}$ by Lemma \ref{lem:kmut}.
Hence, it follows from Theorem \ref{thm:main} that $\Gamma_\e$ and $\Gamma_{\e'}$ are not isomorphic. 
\end{ex}
This example shows that $\type \Gamma_\e =\type \Gamma_{\e'}$ does not imply isomorphism of $\Gamma_\e$ and $\Gamma_{\e'}$.

\section*{Acknowledgments}
The second author thanks Pieter Belmans for helpful communications on point varieties of skew polynomial algebras.
The second author also thanks Izuru Mori for valuable comments on a draft of this paper.

The first author was supported by JSPS Grant-in-Aid for Scientific Research (C) 20K03513.
The second author was supported by JSPS Grant-in-Aid for Early-Career Scientists 18K13381.

\end{document}